\documentclass{amsart}

\usepackage{amsmath,amsthm,amscd,amssymb}
\usepackage{mathrsfs}
\usepackage[dvips]{graphicx}
\usepackage{color}
\usepackage{amsthm}
\usepackage{amsmath}
\usepackage{amscd}
\usepackage{comment}
\usepackage{ascmac}
\usepackage{mathtools}
\usepackage{enumerate}
\usepackage[normalem]{ulem}

\newtheorem{theorem}{Theorem}[section]

\newtheorem{lemma}[theorem]{Lemma}

\theoremstyle{remark}

\usepackage{amssymb}\theoremstyle{definition}
\newtheorem{definition}[theorem]{Definition}

\newtheorem{remark}[theorem]{Remark}
\newtheorem{problem}[theorem]{Problem}

\numberwithin{equation}{section}

\begin{document}
  \title{Arcsine law for random dynamics with a core}

\author{Fumihiko Nakamura}
\address[Fumihiko Nakamura]{Faculty of Engineering, Kitami Institute of Technology, Hokkaido, 090-8507, JAPAN}
\email{nfumihiko@mail.kitami-it.ac.jp}

\author{Yushi Nakano}
\address[Yushi Nakano]{Department of Mathematics, Tokai University,  Kanagawa, 259-1292, JAPAN}
\email{yushi.nakano@tsc.u-tokai.ac.jp}

\author{Hisayoshi Toyokawa}
\address[Hisayoshi Toyokawa]{Faculty of Engineering, Kitami Institute of Technology, Hokkaido, 090-8507, JAPAN}
\email{h\_toyokawa@mail.kitami-it.ac.jp}

\author{Kouji Yano}
\address[Kouji Yano]{Graduate School of Science, Kyoto University, Kyoto 606-8502, JAPAN}
\email{kyanomath@gmail.com}

\makeatletter
\@namedef{subjclassname@2020}{\textup{2020} Mathematics Subject Classification}
\makeatother



    \date{\today}

\begin{abstract}
In their recent paper \cite{HY2021}, G.~Hata and the fourth author first gave an example of random iterations of two piecewise linear interval maps without (deterministic)   indifferent periodic points for which the arcsine law --  a characterization of intermittent dynamics in infinite ergodic theory -- holds. The key in the proof of the result is the existence of a Markov partition preserved by each interval maps. In the present paper, we give a class of random iterations of two    interval maps without indifferent periodic points but satisfying the arcsine law, by introducing a concept of core random dynamics. As applications, we show that the generalized arcsine law holds for generalized Hata--Yano maps and piecewise linear versions of Gharaei--Homburg maps, both of which do not have a Markov partition in general. 
\end{abstract}

  \maketitle
  
\section{Introduction}\label{s:intro}
This paper concerns   the generalized arcsine law of random iterations of interval maps with intermittent behavior.
Intermittency is the irregular alternation of phases of apparently laminar and chaotic dynamics,  and commonly observed in fluid flows  near the transition to turbulence. 
The well-known model of intermittent dynamics is  the Pomeau--Manneville  map, that is,  a piecewise  expanding map $f$  of the interval $[0, 1]$ with two increasing surjective branches and an \emph{indifferent} fixed point at the boundary point $0$, such as 
\begin{equation}\label{eq:1017a}
f(x) =x + x^{p+1} \mod 1 \quad (p>0), 
\end{equation}
named after Pomeau and Manneville since they numerically studied  such interval maps in  \cite{PM1980}.
The existence of  the indifferent fixed point $0$ makes a typical orbit of $f$ have  long stays  around $0$ (the laminar phases) while the expanding property of $f$ makes the orbit have bursts outside  the neighborhood of $0$ (the chaotic phases). 
Furthermore, the   Pomeau--Manneville map  with the indifferent fixed point $0$ of order $p+1$, such as \eqref{eq:1017a},  is known to possess an absolutely continuous invariant measure $\mu$ whose density is of order $x^{-p}$ near $0$, so  $p\geq 1$ if and only if $\mu$ has \emph{infinite} mass near $0$, meaning that  the orbit  stays most of time near the indifferent fixed point $0$.

Thaler showed in \cite{Thaler1983} that when $f$ is a piecewise expanding  interval map with two increasing surjective branches and two  indifferent fixed points at the boundary points $0$ and $1$ of the same order $p+1$ for $p\geq 1$,  such as
\[
f(x) =\begin{cases}
x + 2^p x^{p+1}  \quad &(x\in [0,\frac{1}{2}))\\
x - 2^p (1-x)^{p+1}  \quad &(x\in [\frac{1}{2},1])
\end{cases}
\]
then $f$ has a unique  absolutely continuous invariant measure $\mu$ with infinite mass near $0$ and $1$, and also showed in \cite{Thaler2002} that the arcsine law holds in the sense that
\[
\lim _{N\to \infty}\mu \left( \frac{S_N^+}{N} \leq a \right)  = \frac{2}{\pi}\arcsin \sqrt a\quad \text{with} \quad S_N^+(x) = \sum _{n=0}^{N-1}1_{[\frac{1}{2},1]}\circ  f^n(x)
\]
for each $a\in [0,1]$ under a   condition on  wandering rates (see Theorem \ref{thm:t} for precise description for the condition). 
This was recently generalized in \cite{SY2019} to multi-ray settings   by T.~Sera and the fourth author, by using ideas from excursion theory. 
See also   \cite{AS2019, AS2021, Sera2020, TZ2006} and reference therein.

Although the existence of indifferent fixed points of maps  in the above works is indispensable for the generalized arcsine law of  intermittent dynamics, G.~Hata and the fourth author recently  showed in \cite{HY2021}  that random iterations of two piecewise linear  interval maps  without   indifferent periodic points,
  \begin{equation}\label{eq:HY}
f_0(x) =
\begin{cases}
\frac{x}{2}    \quad &(x\in [0, \frac{1}{2}))\\
2x - 1  \quad &(x\in [\frac{1}{2},1])
\end{cases},
\quad
f_1(x) =
\begin{cases}
2x  \quad &(x\in [0,\frac{1}{2}))\\
\frac{x+1}{2}    \quad &(x\in [\frac{1}{2},1])
\end{cases},
\tag{HY}
  \end{equation}
 where   each map is chosen with probability $\frac{1}{2}$ at each step,  exhibit   the   arcsine law. 
The endpoints $0$ and $1$ are common fixed points of $f_0$ and $f_1$, and   they are in fact \emph{indifferent in average} in the sense that
\begin{equation}\label{eq:1102a}
 \frac{\log\vert f_0'(0) \vert + \log\vert   f_1'(0) \vert}{2} = \frac{\log\vert f_0'(1) \vert + \log\vert f_1'(1)  \vert}{2} =0.
\end{equation}
Such an indifferent in average fixed point was also called zero Lyapunov exponent at the boundary by Gharaei and Homburg \cite{GH2017}: they showed the ``on-off intermittency'' (which may be more easily observed than the generalized arcsine law; 
see Remark \ref{cor:1b} for details)
 for  random iterations of  two  maps $f_0$ and $f_1$ on the interval $[0,1]$ such that 
  \begin{itemize}
  \item[(GH1)]
both $f_0$ and $f_1$ are  diffeomorphisms,
\item[(GH2)] $f_0(x) <x$ and $f_1(x)>x$ for all $x\in [0,1]$,
\item[(GH3)] $f_0(0)=f_1(0) =0$, $f_0(1)=f_1(1)=1$ and \eqref{eq:1102a} holds.
\end{itemize}
See Figure \ref{fig1}. 
This was generalized to chaotically driven non-autonomous iterations of such a pair of two maps in \cite{HR2020}. 
See also  \cite{AGH2018, HKRVZ2021, HP2019} for recent works for critical intermittency of iterated functions systems.

\begin{figure}[hbt]
\centering
\includegraphics[bb=0 0 650 260, width=12cm]{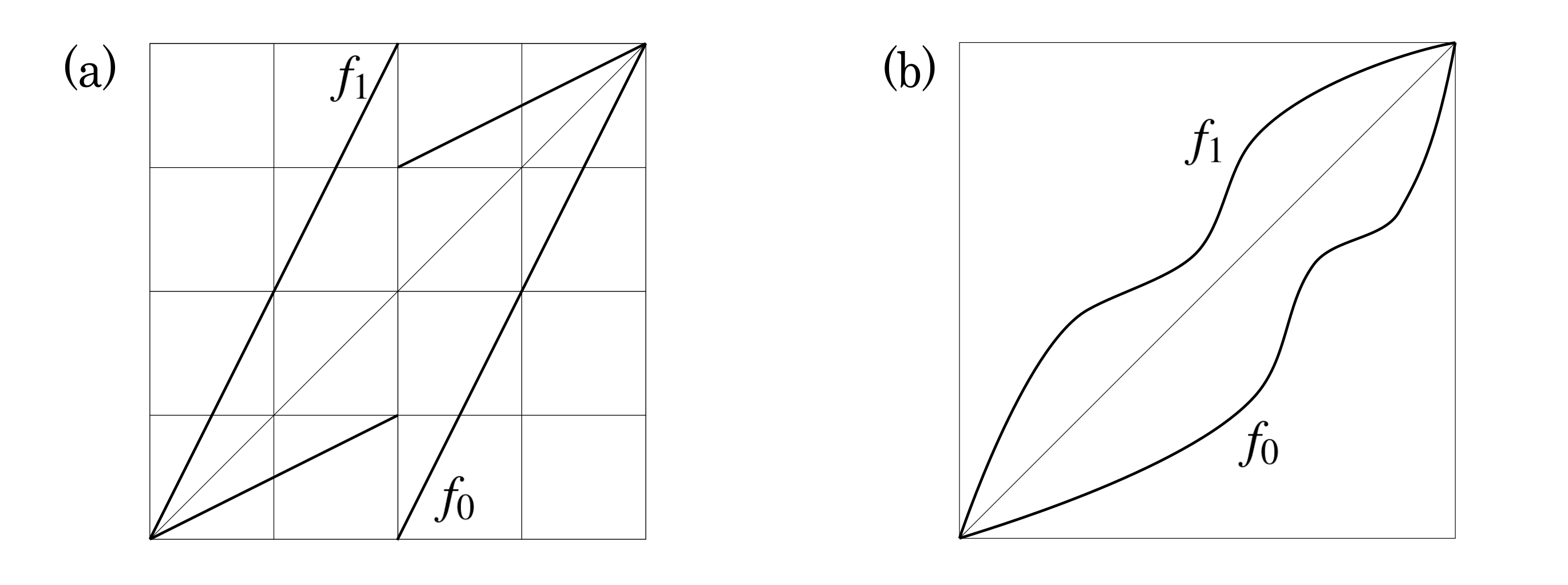}
\caption{(a) Hata--Yano  map; (b) Gharaei--Homburg  map }
\label{fig1}
\end{figure}

The key in the proof of the arcsine law in \cite{HY2021} is the existence of a Markov partition  preserved by each interval maps.
In the present paper, we give a  class of  random iterations of two    interval maps  without   indifferent periodic points but  satisfying the   arcsine law, by introducing a concept of core dynamics. 
As applications, we show that the  arcsine law holds for some generalized Hata--Yano (HY) maps and some piecewise linear versions 
 of   Gharaei--Homburg (GH) maps, both of which do not have a Markov partition in general (see Figure \ref{fig1b} as examples 
  and refer to Section \ref{ss:example} for precise definitions). 
  
\begin{figure}[hbt]
\centering
\includegraphics[bb=0 0 650 260, width=12cm]{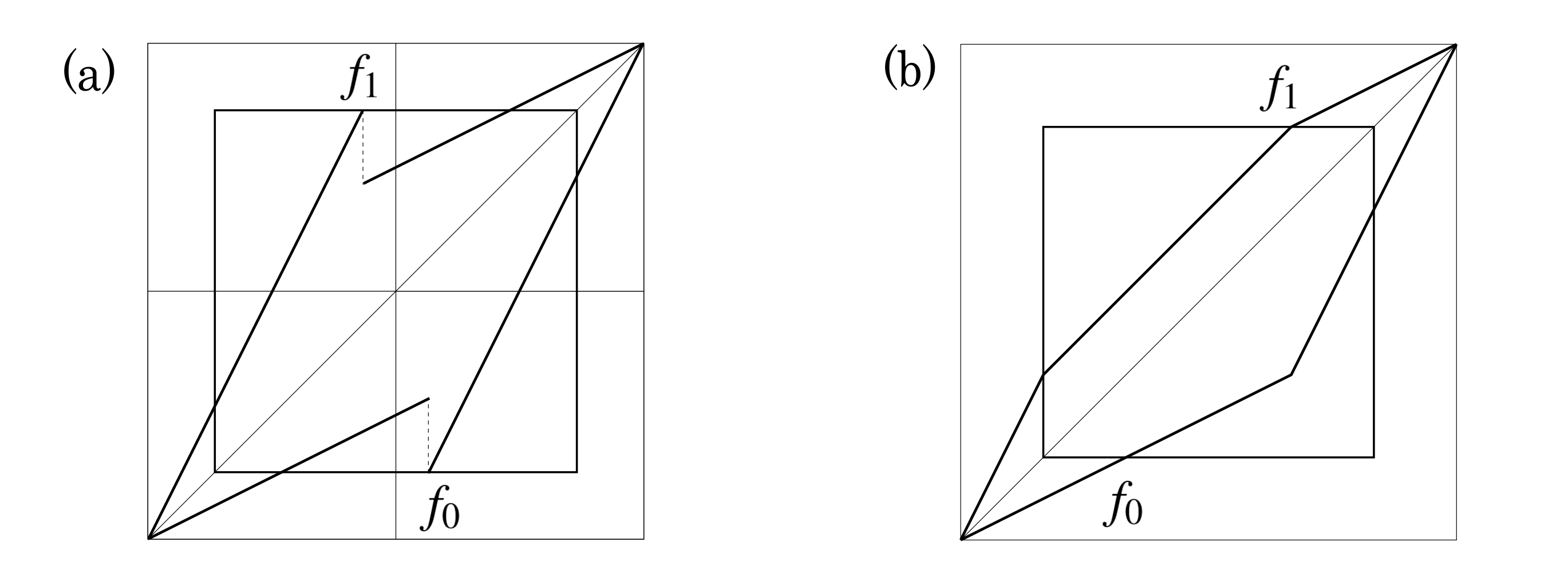}
\caption{(a) Generalized HY  map; (b) Piecewise linear GH  map }
\label{fig1b}
\end{figure}

\subsection{Main result: core  random dynamics}

We consider two interval maps $f_0$, $f_1:[0,1]\to[0,1]$ such that there are a real number $c\in (0,\frac{1}{2}]$ and measurable maps $g_0:[c,1-\frac{c}{2}]\to [\frac{c}{2},1-c]$,  $g_1:[\frac{c}{2},1-c]\to[c,1-c]$ satisfying that 
\begin{eqnarray}\label{eq:1103}
\qquad
f_0(x)=\begin{cases}
\frac{x}{2} & x\in[0,c)\\
g_0(x) & x\in[c,1-\frac{c}{2})\\
2x-1 & x\in[1-\frac{c}{2},1]
\end{cases},
\quad
f_1(x)=\begin{cases}
2x & x\in[0,\frac{c}{2})\\
g_1(x) & x\in[\frac{c}{2},1-c)\\
\frac{x+1}{2} & x\in[1-c,1]
\end{cases}.
\end{eqnarray}
See Figure \ref{fig2}.
Note that given $(f_0, f_1)$ satisfying \eqref{eq:1103} with some $(c, g_0, g_1)$, one may find another pair $(\tilde c, \tilde g_0, \tilde g_1)$  for which  \eqref{eq:1103} holds instead of $(c, g_0, g_1)$.
Notice that, if we take $g_0(x) =\frac{x}{2}$ on $[c,\frac{1}{2}+\frac{c}{4})$, $2x-1$ on $[\frac{1}{2}+\frac{c}{4}, 1-\frac{c}{2}]$ and $g_1(x) =2x$ on $[\frac{c}{2},\frac{1}{2}-\frac{c}{4})$, $\frac{x+1}{2}$ on $[\frac{1}{2}-\frac{c}{4}, 1-c]$, then $(f_0, f_1)$ converges to the one given in (HY) in the limit $c\to 0$. Notice also that 
$(f_0, f_1)$ satisfies (GH1)-(GH3) when $g_0$, $g_1$ are diffeomorphisms,   $f_0$ is of  class $\mathcal C^1$ near $x=c$ and $x=1-\frac{c}{2}$ and  $f_1$ is of class $\mathcal C^1$ near $x=\frac{c}{2}$ and $x=1-c$. 
As one can see below, the slopes $2$ and $\frac{1}{2}$ are not essential in our main result: they can be replaced by $\lambda$ and $\lambda ^{-1}$ for any $\lambda >1$.
We emphasize that, in contrast,  Hata--Yano \cite{HY2021} essentially required that the slopes be $2$ and $\frac{1}{2}$ in order to ensure the existence of a Markov partition.

Let $T$ be a random map of $[0,1]$ such that $T=f_0$ and $f_1$ with equal probabilities.

\begin{figure}[hbt]
\centering
\includegraphics[bb=0 0 650 260, width=12cm]{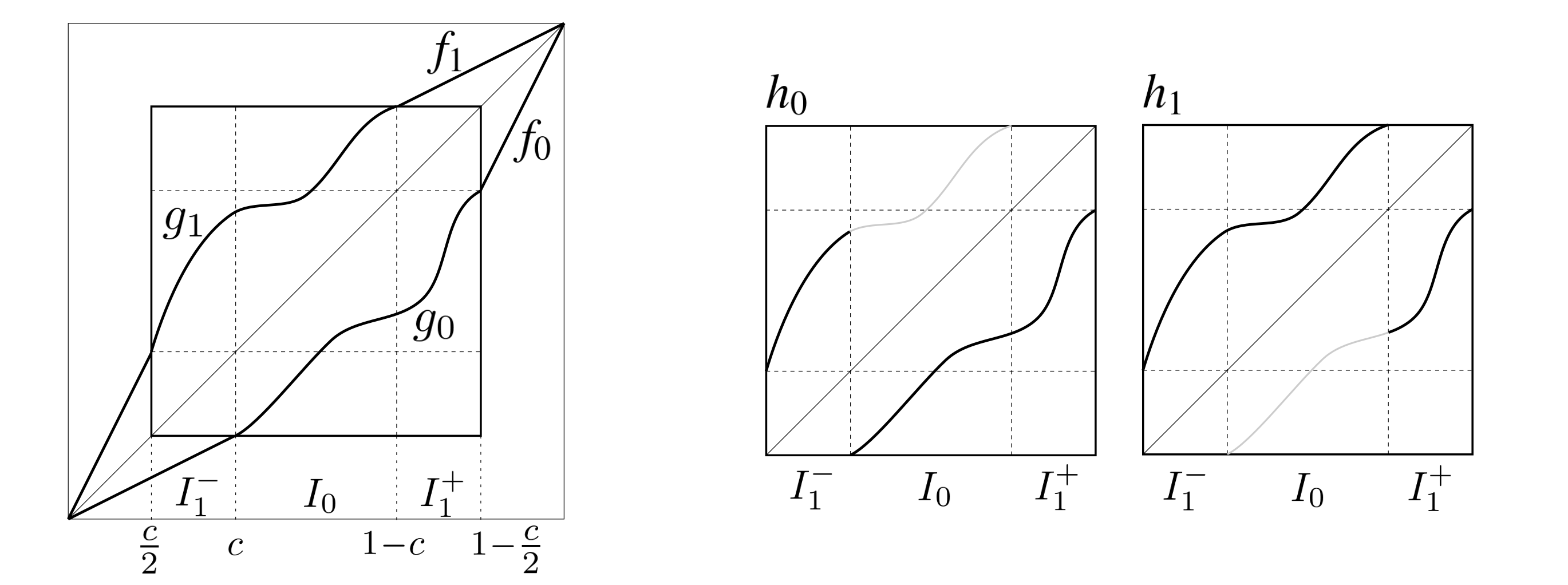}
\caption{Random dynamics   with a core random dynamics }
\label{fig2}
\end{figure}

For notational simplicity, we denote $[\frac{c}{2}, c)$, $[c, 1-c)$, $[1-c, 1- \frac{c}{2})$ by $I_1^-$, $I_0$, $I_1^+$, respectively, and $I_1^- \cup I_0 \cup I_1^+$ by $Y$.
Define two maps $h_0$, $h _1: Y\to Y$ given by
\begin{eqnarray}\label{eq:1103b}
\qquad
h _0(x)=\begin{cases}
g_1(x) & x\in I_1^-\\
g_0(x) & x\in I_0 \cup I_1^+
\end{cases},
\quad
h _1(x)=\begin{cases}
g_1(x) & x\in I_1^- \cup I_0\\
g_0(x) & x\in I_1^+
\end{cases},
\end{eqnarray}
and let $T_{\mathrm{core}}$ be a random map of $Y$ such that $T_{\mathrm{core}}= h _0$ and $h_1$ with equal probabilities.

\begin{definition}\label{def:1}
Let $\{T_n\}_{n=1}^\infty$ be an i.i.d.~sequence of random maps whose distribution is the same as that of $T$, 
 and let
$$
T^{(n)}=T_n\circ T_{n-1}\circ\cdots\circ T_1
$$
for each $n\geq 1$.
Similarly, we define $T^{(n)}_{\mathrm{core}}$ for $n\geq 1$ from $T_{\mathrm{core}}$. 
Then, we call $\{ T^{(n)}\} _{n=1}^\infty$  a \emph{random dynamics   with a core random dynamics $\{ T^{(n)}_{\mathrm{core}}\} _{n=1}^\infty$ }.
\end{definition}

Recall that,  given a random map $S$ on an interval $I$ with two measurable maps $h_0$, $h_1$ on $I$, that is, $S=h_0$ and $h_1$ with equal probabilities,  a measure $\nu$ on $I$ is called \emph{$S$-invariant}   if $\nu$ is not a zero measure and 
\[
\frac{(h_0)_* \nu +(h _1)_*\nu}{2}=\nu ,
\]
where $(h_j)_* \nu = \nu \circ h _j^{-1}$ for $j=0, 1$.
Moreover, an $S$-invariant measure $\nu$ is called \emph{annealed metrically transitive}
 if for any two Borel sets $A, B \subset I$ with $\nu (A) \nu (B) > 0$, there exists an integer $n\geq 0$   such that $\nu (S ^{ (-n)} A \cap  B) $ has a positive expectation, where $S^{(-n)} A$ is the inverse image of $A$ by the random composition $S^{(n)}$  of an i.i.d.~sequence of $n$ random maps whose distribution is same as the one of $S$.

We can now state   our main result.
\begin{theorem}\label{thm:main}
Let $\{T^{(n)}\}_{n=1}^\infty$ be a random dynamics   with a core random dynamics $\{ T_{\mathrm{core}}^{(n)} \}_{n=1}^\infty$ given in Definition \ref{def:1}. 
Suppose that there exists an annealed metrically transitive $T_{\mathrm{core}}$-invariant probability measure $\nu$ with $\nu (I_1^-) \nu (I_1^+) > 0$.
Then, for any random variable $\Theta$ with values in $Y$  which is independent of the random maps $\{T_n\} _{n=1}^\infty$ and whose distribution is absolutely continuous with respect to  $\nu$, 
  it holds that
\begin{eqnarray}\label{eq:1103c}
\frac{1}{N}\sum_{n=0}^{N-1}1_{\{T^{(n)}(\Theta)\geq \frac{1}{2}\}}
\xrightarrow[N\to\infty]{\mathrm{d}}
\frac{1}{\pi\sqrt{x(1-x)}} \cdot\frac{b}{b^{2} x+(1-x)}d x 
\end{eqnarray}
with
\[
 b=\frac{1-\beta}{\beta}, \quad \beta=\frac{\nu(I_1^-)}{\nu(I_1^-)+\nu(I_1^+)}, 
\]
where we mean by $\xrightarrow[N\to\infty]{\mathrm{d}}$ the convergence in distribution.
\end{theorem}

\begin{remark}
When $\nu(I_1^-) = \nu(I_1^+) $,   we have $\beta =\frac{1}{2}$ and $b=1$,  \emph{independently of the value of $\nu (I_0)$.}
 Thus, it follows from \eqref{eq:1103c} that in the case we have
\[
\lim _{N\to \infty} {\rm Prob}\left( \frac{1}{N} \sum_{n=0}^{N-1}1_{\{T^{(n)}(\Theta)\geq\frac{1}{2}\}} \leq a \right) = \int ^a _{0} \frac{1}{\pi\sqrt{x(1-x)}}  d x  = \frac{2}{\pi } \arcsin \sqrt a
\]
for each $a \in [0,1]$.
Hence, \eqref{eq:1103c} is called the generalized arcsine law. 
\end{remark}

\begin{remark}\label{rem:DK}
We show Theorem \ref{thm:main} by applying Thaler--Zweim{\"u}ller's abstract generalized arcsine law  (\cite[Theorem 3.2]{TZ2006}) to the skew-product induced by $T$ (which we will recall in Section \ref{ss:TZ}).
In the same paper, they also show the Darling--Kac law (\cite[Theorem 3.1]{TZ2006}), so 
 by combining it with the estimates  in the proof of Theorem \ref{thm:main} for wandering rates, we can show the Darling--Kac law for random dynamics with a  core random dynamics: under the assumption of Theorem \ref{thm:main}, it holds that
\[
\frac{1}{\sqrt N}\sum_{n=0}^{N-1}1_{\{T^{(n)}(\Theta)\in E\}}
\xrightarrow[N\to\infty]{\mathrm{d}}
 \frac{2\mu(E)}{\mu(E_0)}
 \vert \mathcal N\vert , 
\]
%or
%\[
%\frac{1}{\sqrt N}\sum_{n=0}^{N-1}1_{\{T^{(n)}(\Theta)\in A\}}
%\xrightarrow[N\to\infty]{\mathrm{d}}
% \frac{2\nu(I_1^-\cap A) + \nu(I_0 \cap A) + 2\nu(I_1^+\cap A) }%{\nu(I_1^-)+\nu(I_1^+)}
% \vert \mathcal N\vert , 
%\]
for any $E\subset Y$, 
where $E_0:=I_1^-\cup I_1^+$ and $\mathcal N$ is a random variable with standard normal distribution, where $\mu$ is a $T$-invariant infinite measure naturally derived from $\nu$ (see \eqref{inv2-2} for its precise definition). Refer to  Remark \ref{rem:211} for more details.
\end{remark}

\begin{remark}\label{cor:1}
As an easy consequence  
 from   Theorem \ref{thm:main}, 
we  finally 
remark 
 the following \emph{pointwise}   generalized arcsine law: 
 Assume that $\{T^{(n)}\}_{n=1}^\infty$ and $\nu$  satisfy the conditions in Theorem \ref{thm:main}. 
 Assume also that $\nu$ is a \emph{discrete} measure.
  Then, for any $y$ in the support of $\nu$, it holds that
\begin{equation}\label{eq:0326b}
\frac{1}{N} \sum_{n=0}^{N-1}1_{\{T^{(n)}(y)\geq\frac{1}{2}\}} \xrightarrow[N\to\infty]{\mathrm{d}} \frac{1}{\pi\sqrt{x(1-x)}} \cdot\frac{b}{b^{2} x+(1-x)}d x. 
\end{equation}
 (Apply Theorem \ref{thm:main} to $\Theta$ whose   distribution   is the Dirac measure at $y$, which is absolutely continuous with respect to $\nu$ since $\nu$ is discrete.) 
This observation may be useful to establish 
  \eqref{eq:0326b} for each $y$ in a large subset of $Y$ 
as demonstrated for the example in Section \ref{sss:gHY}, but \eqref{eq:0326b} may hold beyond the above setting. 
In fact, for the random dynamics in 
Section \ref{sss:pGH}, 
 only a finite set of $Y$ can be supported by an invariant probability measure $\nu$ but \eqref{eq:0326b} holds on a nontrivial interval (i.e.~an interval whose interior is nonempty) of $Y$. 
\end{remark}
\begin{remark}\label{cor:1b}
We  recall that Gharaei and Homburg \cite{GH2017} said that    the \emph{on-off intermittency} holds for $T$ if,   for any sufficiently small neighborhood $U$ of $0$ and $1$ and any $y\in (0,1)$, it  almost surely holds that 
\begin{align*}
\lim _{n\to\infty} \frac{1}{N}  \sum _{n=0}^{N-1}1_{\{T^{(n)}(y)\in U\}} =1
 \quad \text{and}
\quad
\lim _{N\to\infty}   \sum _{n=0}^{N-1}1_{ \{  T^{(n)}(y) \not\in   U\}}=\infty .
\end{align*}
From them, 
 we have $ \frac{1}{N}   \sum _{n=0}^{N-1}1_{ \{  T^{(n)}(y) \not\in   U\}} \to 0$ as $N\to \infty$.
 That is, both the scales $1$ and $N$ 
 are not nice to understand the long time behavior of $ \sum _{n=0}^{N-1}1_{ \{  T^{(n)}(y) \not\in   U\}}$, 
and 
 the Darling--Kac law in Remark \ref{rem:DK} tells the appropriate scale $\sqrt N$ together with its  limit distribution $\vert \mathcal N\vert$.
Furthermore, 
since Theorem \ref{thm:main} with ``$\{ T^{(n)}(\Theta )\in A\}$'' instead of ``$\{ T^{(n)}(\Theta )\geq \frac{1}{2}\}$'' holds for any interval $A$ including $1$ but not including $0$ (as one can see from the proof of Theorem \ref{thm:main}),  
by splitting $U$ into an interval 
  including $1$ but not including $0$ and 
its complement, 
one would see that the generalized arcsine law is also a  (much) stronger   limit theorem 
 than   the on-off intermittency 
of Gharaei--Homburg.
\end{remark}

\begin{remark}
As mentioned before, Theorem 1.2 can be easily generalized to random dynamics given by (1.3) with slopes $\lambda$ and $\lambda^{-1}$ (for any $\lambda >1$) instead of the slopes $2$ and $2^{-1}$ of (1.3). On the other hand, it is impossible to relax the condition $P(T=f_0) = P(T=f_1)=\frac{1}{2}$ because, if the condition is violated, say $P(T=f_0) < \frac{1}{2}$, then $\frac{1}{N} \sum_{n=0}^{N-1} 1_{\{ T^{(n)}(y) \ge \frac{1}{2}\}}$ almost surely converges to $1$ as $N\to \infty$; see (the proof of) \cite[Theorem 5.2]{GH2017}. Furthermore, the linearities of $f_0$ and $f_1$ around $0$ and $1$ are indispensable in our argument (see e.g.~(3.2)), so it is unclear whether the arcsine law still holds in the case when $f_1$ is nonlinear, say $f_1(x)= 2x + x^2$, near $x=0$.
\end{remark}

\subsection{Examples}\label{ss:example}

In this subsection, we give   examples that satisfy (or do not satisfy)  the hypothesis for the core random dynamics $T_{\mathrm{core}}$ of Theorem \ref{thm:main}.
\subsubsection{Core deterministic dynamics}
The simplest example is a random dynamics with a  deterministic core dynamics, that is, a random core dynamics satisfying  $h_0 =h_1$.
Notice that when $c= \frac{1}{2}$, we automatically have the case.
In such a deterministic   case, according to classical results for deterministic piecewise smooth interval maps, one can find several type of core dynamics    with and without   invariant probability measures. 
For example, when $h_0$ is a topologically mixing $\mathcal C^2$ piecewise expanding map on $Y$, then $h_0$ (and thus $T_{\mathrm{core}}$) has a unique absolutely continuous ergodic invariant probability measure $\nu$ whose support is $Y$ (cf.~\cite{LY1973}).
Furthermore, the measure $\nu$ may have different weights on $I_1^-$ and $I_1^+$ (e.g.~$c= \frac{1}{2}$, $h_0(x) = 2x- \frac{1}{4}$ on $I_1^-$ and $h_0(x)=x - \frac{1}{4}$ on $I_1^+$).
Moreover, since $h_0$ is not necessarily continuous, it is even possible that there exists no invariant probability measure for $h_0$  (e.g.~$c= \frac{1}{2}$, $h_0(x) = \frac{1}{2} x + \frac{1}{4}$ on $I_1^-$ and $h_0(x)=x - \frac{1}{4}$ on $I_1^+$). 
Refer also to   \cite{Blank2017, BPP2019} and references therein for the existence of invariant probability measures of random dynamics.

\subsubsection{Generalized Hata--Yano maps}\label{sss:gHY}
%The next example is a  family of    random maps $T$, for which  the maps  $f_0$ and $f_1$ given by 
We call the random map $T$ a \emph{generalized Hata--Yano map} if,
for a constant $0 \le \delta \le 1/6$, the maps $f_0$ and $f_1$ given by
  \[
f_0(x) =
\begin{cases}
\frac{x}{2}    \quad &(x\in [0, \frac{1}{2}+\delta ))\\
2x - 1  \quad &(x\in [\frac{1}{2}+\delta ,1])
\end{cases},
\quad
f_1(x) =
\begin{cases}
2x  \quad &(x\in [0,\frac{1}{2}-\delta ))\\
\frac{x+1}{2}    \quad &(x\in [\frac{1}{2}-\delta ,1])
\end{cases}
  \]
%  are randomly chosen  with the same probability $\frac{1}{2}$, where $0\leq \delta\leq \frac{1}{6}$.  
are randomly chosen with equal probabilities. Note that $(f_0,f_1)$ is exactly (HY) when $\delta=0$,
%Compare these $(f_0, f_1)$ with (HY), and note that $(f_0, f_1)$ is exactly (HY) when $\delta =0$: we call it \emph{generalized Hata--Yano maps}.
and that the random map $T$ admits a core random dynamics $T_{\mathrm{core}}$ for all $\delta \neq 0$  (e.g.~$c=4\delta $ when $\delta\leq \frac{1}{8}$, and $c= \frac{1}{2}$ when $\delta\geq \frac{1}{8}$). 
In fact, when $\delta \geq \frac{1}{8}$, the core random dynamics is deterministic.

For example, 
for $\delta = \frac{1}{8}$  and $c = \frac{1}{2}$,
 as depicted in Figure \ref{fig1c}, one can easily see that  $h_0=h_1$,  $h_0^2(I) = I$ for $I\coloneqq[\frac{3}{8}, \frac{5}{8})$ and $h_0^2 (x) =x$ on $I$, so $\nu := (\delta _y +\delta _{h_0(y)})/2$  is a $T_{\mathrm{core}}$-invariant probability measure satisfying the conditions of Theorem \ref{thm:main} for any $y\in I \cup h_0(I)$. 
By virtue of Remark \ref{cor:1}, pointwise generalized arcsine law \eqref{eq:0326b} holds for any $y\in I \cup h_0(I)$.
Similarly,  for each $\delta \geq \frac{1}{8}$, one can see that  $T$ has a core deterministic dynamics with $h_0^2(x) =x$ for any $x$ close to $\frac{1}{2}$ and (pointwise) generalized arcsine law holds.

\begin{figure}[hbt]
\centering
\includegraphics[bb=0 0 650 260, width=12cm]{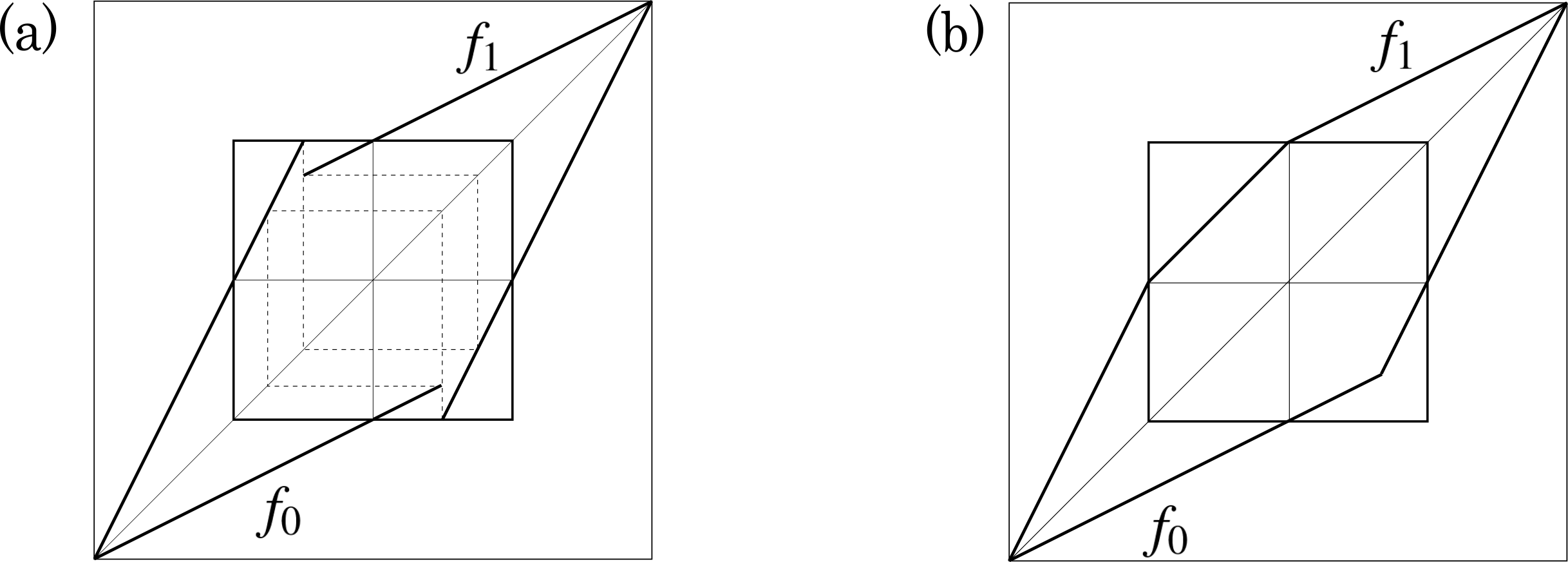}
\caption{(a) Generalized HY  map ($\delta = \frac{1}{8}$); (b) Piecewise linear GH  map (the family \eqref{eq:0329} with $\delta = \frac{1}{2}$)}
\label{fig1c}
\end{figure}

\subsubsection{Piecewise linear versions of Gharaei--Homburg   maps}\label{sss:pGH}
Finally, we consider 
 piecewise linear versions   of 
  Gharaei--Homburg   maps, 
  that is,  a random map $T$ given by a random selection with  probability $\frac{1}{2}$ from $(f_0 , f_1)$ satisfying 
  \begin{itemize}
  \item[($\ell$GH1)] both $f_0$ and $f_1$ are homeomorphisms,  and there are an integer $N_j\geq 2$ and real numbers $c_0^{(j)}=0<c_1^{(j)} < \cdots < c_{N_0}^{(j)}=1$ for $j=0,1$ such that  the restriction of $f_j$ on $[c_{i-1}^{(j)} , c_{i}^{(j)} ]$ has a constant slope for each $j=0, 1$ and $i=1, \ldots , N_j$,
  \end{itemize}
  and (GH2), (GH3).
It is straightforward to see that  
 $T$ 
   always has a core random dynamics (e.g.~for $c=\min \{ c_1^{(0)} , 1- c_{N_1-1} ^{(1)}\}$).  
For simplicity, as before, in the following we assume that the slope of  $f_0$ (resp. $f_1$) is $\frac{1}{2}$ near $x=0$ (resp. $x=1$) and $2$ near $x=1$ (resp. $x=0$).
  
Notice that if $N_j =2$, then  
 $f_j$ must have the form
\begin{equation*} 
f_0(x) =
\begin{cases}
\frac{x}{2}    \quad &(x\in [0, \frac{2}{3}  ))\\
2x - 1  \quad &(x\in [\frac{2}{3},1])
\end{cases},
\quad
f_1(x) =
\begin{cases}
2x  \quad &(x\in [0,\frac{1 }{3}))\\
\frac{x+1}{2}    \quad &(x\in [\frac{1}{3} ,1]),
\end{cases}
\end{equation*}  
which coincides with the generalized Hata--Yano map for $\delta = \frac{1}{6}$,
and thus the analysis of the case $N_0 =N_1=2$ 
 is straightforward.
Indeed, 
 it satisfies that $f_0\circ  f_1(x) =f_1\circ  f_0(x) =x$ everywhere, 
  its  core random dynamics is deterministic  (for $c=\frac{1}{2}$), $h_0$ exchanges $I_1^-=[\frac{1}{4},\frac{1}{2})$ and $I_1^+=[\frac{1}{2},\frac{3}{4})$, and $h_0^2 (x) =x$ on $Y=I_1^-\cup I_1^+$. 
Therefore,   the simplest nontrivial case may be the case when $N_0 =2$ and $N_1=3$, 
such as  the family 
\begin{equation}\label{eq:0329}
f_0(x) =
\begin{cases}
\frac{x}{2}    \quad &(x\in [0, \frac{2}{3}  ))\\
2x - 1  \quad &(x\in [\frac{2}{3},1])
\end{cases},
\quad
f_1(x) =
\begin{cases}
2x  \quad &(x\in [0,\frac{\delta }{2}))\\
x + \frac{\delta}{2}  \quad &(x\in [\frac{\delta }{2}, 1-\delta ))\\
\frac{x+1}{2}    \quad &(x\in [1-\delta ,1])
\end{cases}
\end{equation}
with $0<\delta\leq \frac{2}{3}$, see Figure \ref{fig1c}. 
For $\delta = \frac{1}{2}$, we can take   $c= \frac{1}{2}$ (core deterministic dynamics),   
$h_0$ exchanges $I_1^-$ and $I_1^+$ and the restriction of $h_0^2$ on $I_1^-$ has exactly one sink at the left endpoint $\frac{1}{4}$  whose basin of attraction is $I_1^-$.
Consequently, there exists a unique $T_{\mathrm{core}}$-invariant  probability measure $\nu =(\delta _{\frac{1}{4}} + \delta _{\frac{1}{2}})/2$, which satisfies the conditions of Theorem \ref{thm:main}.
Furthermore, pointwise generalized arcsine law \eqref{eq:0326b} for the  attracting periodic point $p= \frac{1}{4}$ holds by Remark \ref{cor:1}. 
However, unlike the generalized Hata--Yano map with $\delta \geq \frac{1}{8}$, 
Remark \ref{cor:1} does not (directly) tell whether  \eqref{eq:0326b} holds for $y$ in a nontrivial interval 
because there is no   $T_{\mathrm{core}}$-invariant probability measure on $Y$ except $\nu$.
Yet,  one can easily show \eqref{eq:0326b}  for  all $y\in I_1^-$ as follows.  
Let $I_n^- =[\frac{1}{2^{n+1}}, \frac{1}{2^n})$ and $I_n^+= [1- \frac{1}{2^n}, 1- \frac{1}{2^{n+1}})$ for $n\in \mathbb N$ (see also Section \ref{ss:22}). 
Then, it holds that $f_{j_k}\circ f_{j_{k-1}} \circ \cdots \circ f_{j_1}(y)\in I_n^\sigma$ if and only if $f_{j_k}\circ f_{j_{k-1}} \circ \cdots \circ f_{j_1}(p)\in I_n^\sigma$ for any $k \in \mathbb N$,   $(j_1,\ldots ,j_k)\in \{0,1\}^k$, $n\in \mathbb N$ and $\sigma \in \{ +,-\}$ because $y, p\in I_1^-$, $y$ belongs to the basin of attraction   of $p$ and $f_0, f_1$ are monotonically increasing. 
Hence, noticing that $ \bigcup _{n\geq 1}I_n^+ =[\frac{1}{2}, 1)$,  
we get \eqref{eq:0326b} for $y$ 
 from \eqref{eq:0326b} for $p$.

\subsubsection{Problem}\label{ss:problem}
 Finally, we ask a question on   the above examples for  further applications of  Theorem \ref{thm:main}.
\begin{problem}\label{prob:14}
Does there exist  a $T_{\mathrm{core}}$-invariant probability measure satisfying the condition of Theorem \ref{thm:main} for all generalized Hata--Yano maps (except $\delta =0$) and piecewise linear versions of Gharaei--Homburg maps?
\end{problem}
 
Notice  that  in general
 the maps $h_0$ and $h_1$ generating the core random dynamics $T_{\mathrm{core}}$ are not continuous,  even if both the maps $f_0$ and $f_1$ generating $T$ 
  are continuous (see Figure \ref{fig2}). 
Thus, 
even the existence of a $T_{\mathrm{core}}$-invariant probability measure for piecewise linear versions of Gharaei--Homburg maps 
 might be a nontrivial problem.

%\section{Proof}
\section{The Thaler--Zweim{\"u}ller Theorem}\label{ss:TZ}

The proof of Theorem \ref{thm:main} is based on the Thaler--Zweim{\"u}ller theorem, which gives a criterion for the generalized arcsine law of abstract infinite ergodic systems.
In this subsection we briefly recall the Thaler--Zweim{\"u}ller theorem.

In this subsection, let $\tau$ be a conservative, ergodic and measure preserving transformation over a $\sigma$-finite and infinite measure space $(X,\mathcal{B},m)$.
Here we mean by {\it conservative} that any set $W\in\mathcal{B}$ with $\tau^{-k}W\cap \tau^{-l}W=\emptyset\pmod{m}$ for all distinct $k,l\in\mathbb{N}$ is a $\mu$-null set and by {\it ergodic} that any set $E\in\mathcal{B}$ with $\tau^{-1}E=E\pmod{m}$ satisfies $m(E)=0$ or $m(X\setminus E)=0$.
First we recall the definition of the Perron--Frobenius operator of $\tau$ (cf.~\cite{Aaronson1997}).
\begin{definition}
\emph{The Perron--Frobenius operator} $\mathcal{L}_{\tau}:L^1(X,m)\to L^1(X,m)$ corresponding to $\tau$ with respect to $m$ is   defined by
\begin{align*}
\mathcal{L}_{\tau}\phi=\frac{d((\phi\cdot m )\circ \tau^{-1})}{dm}\quad\text{for }\phi\in L^1(X,m)
\end{align*}
where $\phi\cdot m$ is a signed measure given by $(\phi\cdot m)(A)=\int_{A}\phi\, dm$ for $A\in \mathcal B$.
\end{definition}
The Perron--Frobenius operator $\mathcal{L}_{\tau}$ is also characterized by
\begin{align*}
\int_X \mathcal{L}_{\tau}\phi \cdot \psi \, dm=\int_X\phi \cdot \psi \circ \tau\, dm
\end{align*}
for $\phi \in L^1(X,m)$ and $\psi\in L^{\infty}(X,m)$.

We then recall some necessary definitions for stating the Thaler--Zweim\"{u}ler theorem.
In what follows, let $A\in \mathcal B$ be a set of positive and finite $m$-measure.
Recall that {\it the first return time} for $A$ is a map $\varphi=\varphi_{A}: A\to\mathbb{N}\cup\{\infty\}$ defined by
$\varphi(x)\coloneqq\inf\left\{n\in\mathbb{N} : \tau ^nx\in A\right\}$
for $x\in A$, where $\inf\emptyset=\infty$.
When $\tau$ is conservative and ergodic, $\varphi$ can be naturally extended to the function defined and finite $m$-almost everywhere, which is also written by the same symbol $\varphi=\varphi_{A}$ and is referred to be the first return time for $A$.

\begin{definition}
A measurable function $H\ge 0$ supported on $A$ is called {\it uniformly sweeping} for $A$ if there is some $K\in\mathbb{N}_0$ such that
\begin{align*}
\operatorname*{ess\ inf} _{x\in {A}} \sum_{k=0}^{K} \mathcal{L}_{\tau}^{k} H(x)>0. 
\end{align*}
For mutually disjoint sets $A^-,A,A^+\in\mathcal B$, we say that a set ${A}$ {\it dynamically separates $A^-$ and $A^+$} 
if $\tau^k x\in A^-$ and $\tau^l x\in A^+$ for some $k,l\in\mathbb{N}$ imply $\tau^nx\in {A}$ for some $n\in\mathbb{N}$ with $k<n<l$ or $l<n<k$.
For a set ${A}$, {\it the wandering rate} is defined by
\begin{align*}
w_{N}({A})\coloneqq\sum_{n=0}^{N-1} m\left({A} \cap\{\varphi>n\}\right)=\int_{{A}}\left(\sum_{n=0}^{N-1} \mathcal{L}_{\tau}^{n} 1_{A_{n}}\right) dm
\end{align*}
where $A_{0}\coloneqq A$ and $A_{n}\coloneqq A^{c} \cap\{\varphi=n\}$ ($n \in \mathbb{N}$),
and when $A$ dynamically separates $A^-$ and $A^+$ we also define
\begin{align*}
w_{N}\left(A, A^{\pm}\right)&\coloneqq\sum_{n=0}^{N-1} m\left(A \cap \tau^{-1} A^{\pm} \cap\{\varphi>n\}\right)\\
&=m\left(A \cap \tau ^{-1} A^{\pm}\right)+\sum_{n=1}^{N-1} m\left(A_{n} \cap A^{\pm}\right).
\end{align*}
\end{definition}

Now we recall the Thaler--Zweim\"{u}ler theorem.
Recall that the \emph{strong distributional convergence} of random variables $\{ L_N\} _{N\in \mathbb N}$ on a measure space $(X,\mathcal B,m)$ to a random variable $R$, denoted by $L_N\xrightarrow[N\to\infty]{\mathrm{d}(m)} R$, means that   $\{ L_N\} _{N\in \mathbb N}$ converges to $R$ in distribution with respect to any probability measure $m_1$ which is absolutely continuous with respect to $m$. 

\begin{theorem}[\cite{TZ2006}]\label{thm:t}
Let $\tau$ be a conservative, ergodic and measure preserving transformation on a $\sigma$-finite and infinite measure space $(X,\mathcal B, m)$.
Suppose the following conditions:
\begin{itemize}
\item[(C1)]
There is a set $A\in\mathcal B$ with $0<m(A)<\infty$ such that 
\begin{align*}
\lim_{N\to\infty}\frac{1}{w_{N}(A)} \sum_{n=0}^{N-1} \mathcal{L}_{\tau}^{n} 1_{A_{n}} = H \quad \text {uniformly on } A
\end{align*}
for some $H:A\to [0,\infty)$ uniformly sweeping.
\item[(C2)] 
$w_{(\cdot)}(A):\mathbb{N}\to\mathbb{R}_+$ is regularly varying with exponent $1-\alpha>0$ i.e., $w_N(A)=N^{1-\alpha}\cdot\psi(N)$ for some $\psi:\mathbb{R}_+\to\mathbb{R}_+$ such that for any $\lambda>0$ we have
\begin{align*}
\lim_{x\to\infty}\frac{\psi(x\lambda)}{\psi(x)}=1.
\end{align*}
\item[(C3)]
There is a partition $X=A^-\cup A\cup A^+$ with $m(A^-)=m(A^+)=\infty$ such that $A$ dynamically separates $A^-$ and $A^+$, and
\begin{align*}
\lim_{N\to\infty}\frac{1}{w_{N}\left(A, A^-\right)} \sum_{n=0}^{N-1} \mathcal{L}_{\tau}^{n} 1_{A_{n} \cap A^-} = H^{-} \quad \text { uniformly on } A
\end{align*}
for some $H^-:A\to[0,1)$ uniformly sweeping.
\item[(C4)]
There exists some $\beta\in(0,1)$ such that
\begin{align*}
\lim_{N\to\infty}\frac{w_{N}\left(A, A^-\right)}{w_{N}(A)} = \beta.
\end{align*}
\end{itemize}
Then, the strong distributional convergence
\begin{align*}
\frac{1}{N} \sum_{n=0}^{N-1} 1_{B} \circ \tau^{n}
\xrightarrow[N\to\infty]{\mathrm{d}(m)}
\frac{b \sin \pi \alpha}{\pi}  \frac{x^{\alpha-1}(1-x)^{\alpha-1}}{b^{2} x^{2 \alpha}+2 b x^{\alpha}(1-x)^{\alpha} \cos \pi \alpha+(1-x)^{2 \alpha}} d x 
\end{align*}
holds for any $B\in\mathcal B$ satisfying $\mu(B\triangle A^-)<\infty$, where $b\coloneqq(1-\beta)/\beta$.
\end{theorem}

%\section{Proof}
\section{$T$-invariant measure}\label{ss:22}

In this section, as a preliminary for the proof of Theorem \ref{thm:main}, we construct a $T$-invariant measure $\mu$ from the   $T_{\mathrm{core}}$-invariant measure $\nu$ of Theorem \ref{thm:main}.
We first introduce a partition $\zeta$ of $[0,1)$ by
\begin{align*}
\zeta=\left\{I_n^-, I_0, I_n^+\ :\ n\in\mathbb{N}\right\}
\end{align*}
where we set, for each $n\in\mathbb{N}$,
\begin{align*}
I_{n}^-=\left[\frac{c}{2^{n}},\frac{c}{2^{n-1}}\right),\quad
I_0=\left[c,1-c\right),\quad
I_{n}^+=\left[1-\frac{c}{2^{n-1}},1-\frac{c}{2^{n}}\right).
\end{align*}

In order to construct a $\sigma$-finite invariant measure for the entire random dynamical system $T$, we assume that a partial random dynamical system $T_{\mathrm{core}}$ on $I_1^-\cup I_0\cup I_1^+$ with $\mathbb{P}(T_{\mathrm{core}}= h _0) =\mathbb{P}(T_{\mathrm{core}} =h _1) = \frac{1}{2}$ admits an invariant probability measure $\nu$ supported on $I_1^-\cup I_0\cup I_1^+$, that is,
\begin{align}\label{invariance2-1}
 (f_1)_* \nu_- +  \frac{1}{2}(f_0)_* \nu_0 +  \frac{1}{2}(f_1)_* \nu_0 +  (f_0)_* \nu_+ =\nu_- + \nu_0 + \nu_+,
\end{align}
where $\nu_-(\,\cdot\,)=\nu(\,\cdot\cap I_1^-)$, $\nu_0(\,\cdot\,)=\nu(\,\cdot\cap I_{0})$ and $\nu_+(\,\cdot\,)=\nu(\,\cdot\cap I_1^+)$.
If we set a measure $\mu$ on $[0,1]$ as
\begin{align}\label{inv2-2}
\mu=\begin{cases}
2(f_0)_*^{n-1}\nu_- & \text{on $I_{n}^-$ ($n\in\mathbb{N}$),}\\
\nu_0 & \text{on $I_{0}$},\\
2(f_1)_*^{n-1}\nu_+ & \text{on $I_{n}^+$ ($n\in\mathbb{N}$)},
\end{cases}
\end{align}
then $\mu$ is a $\sigma$-finite invariant measure for $T$ which will be shown in the next lemma.
%Note that $\nu$ is not necessarily absolutely continuous with respect to the Lebsgue measure on $[0,1]$ and neither is the resulting invariant measure $\mu$ when $\nu$ is not.
Note that $\nu$ and hence $\mu$ are possibly singular to the Lebesgue measure on $[0,1]$.

\begin{lemma}\label{lem2-3}
Let $\nu$ be an invariant Borel probability measure for $T_{\mathrm{core}}$.
Then $\mu$ given by the equation (\ref{inv2-2}) is a $\sigma$-finite $T$-invariant Borel measure 
for which we have
$\mu([\epsilon,1-\epsilon])<\infty$ for any $\epsilon\in(0,\frac{1}{2})$.

Moreover,
if $\nu (I_1^-)\nu (I_1^+)>0$ holds then
$\mu([0,\epsilon))=\mu((1-\epsilon,1])=\infty$ for any $\epsilon\in(0,\frac{1}{2})$ and hence $\mu$ is a $\sigma$-finite and infinite measure.
\end{lemma}

\begin{proof} 
We check the statement by a direct calculation.
For any $A\subset I_n^-$ where $n\ge 2$, we have $f_1^{-1}A\subset I_{n+1}^-$ and $f_0^{-1}A\subset I_{n-1}^-$.
Then from the construction of $\mu$ in (\ref{inv2-2}),
\begin{align*}
\mathbb{E}\left[\mu\left(T^{-1}A\right)\right]
&=\left(\frac{1}{2}(f_1)_*\mu +\frac{1}{2}(f_0)_*\mu\right)(A)\\
&=\frac{1}{2}\left( \mu\left(f_1^{-1}A\right)+\mu\left(f_0^{-1}A\right) \right)\\
&=\frac{1}{2}\left( 2\nu_-\left(f_0^{-n}\left(f_1^{-1}A\right)\right)+2\nu_-\left(f_0^{-(n-2)}\left(f_0^{-1}A\right)\right) \right)\\
&=\nu_-\left(f_0^{-(n-1)}A\right)+\nu_-\left(f_0^{-(n-1)}A\right)\\
&=\mu(A).
\end{align*}
In the second last equality, we used the fact that $f_0^{-1}=f_1$ on $I_n^-$ for $n\ge2$.
For the case on $I_n^+$ ($n\ge2$), we can calculate by the same manner and we omit it.
On $I_1^-\cup I_0\cup I_1^+$, we also have for any $A\subset I_1^-\cup I_0\cup I_1^+$,
\begin{align*}
\mathbb{E}\left[\mu\left(T^{-1}A\right)\right]
&=\left(\frac{1}{2}(f_1)_*\mu +\frac{1}{2}(f_0)_*\mu\right)(A)\\
&=\frac{1}{2}\bigg( \mu\left(f_1^{-1}A\cap I_1^-\right) +\mu\left(f_1^{-1}A\cap I_0\right) +\mu\left(f_1^{-1}A\cap I_2^-\right)\\
&\qquad\qquad \mu\left(f_0^{-1}A\cap I_1^+\right) +\mu\left(f_0^{-1}A\cap I_0\right) +\mu\left(f_0^{-1}A\cap I_2^+\right) \bigg)\\
&=\frac{1}{2}\bigg( 2\nu_-\left(f_1^{-1}A\right) +\nu_0\left(f_1^{-1}A\right) +2\nu_-\left(f_0^{-1}\left(f_1^{-1}A\right)\right)\\
&\qquad\qquad 2\nu_+\left(f_0^{-1}A\right) +\nu_0\left(f_0^{-1}A\right) +2\nu_+\left(f_1^{-1}\left(f_0^{-1}A\right)\right) \bigg).
\end{align*}
From the equation (\ref{invariance2-1}) and that $f_0$ and $f_1$ are inverse maps to each other on $I_n^{\pm}$ ($n\ge2$), we get
\begin{align*}
\mathbb{E}\left[\mu\left(T^{-1}A\right)\right]
&=\nu_-(A)+\nu_0(A)+\nu_+(A) +\nu_-(A)+\nu_+(A)\\
&=\mu(A)
\end{align*}
and hence $\mu$ is $T$-invariant.

By the equation (\ref{inv2-2}), we have
\begin{align*}
\mu(I_{n+1}^-)=2\nu_-\left(f_0^{-n}I_{n+1}^-\right)=2\nu_-\left(f_0^{-n}\left(f_1^{-n}I_{1}^-\right)\right)=2\nu_-(I_1^-)<\infty
\end{align*}
and similarly $\mu(I_{n+1}^+)=2\nu_+(I_1^+)<\infty$ for any $n\in\mathbb{N}$.
Then the other claims are valid.
\end{proof}

\begin{remark}
When a $T_{\mathrm{core}}$-invariant probability measure $\nu$ is Lebesgue-absolutely continuous, we can have the density function of the resulting $\sigma$-finite $T$-invariant measure $\mu$:
let $\varphi_-$, $\varphi_0$ and $\varphi_+$ be the densities of $\nu_-$, $\nu_0$ and $\nu_+$, respectively.
Then the density fucntion of $\mu$ can be represented as
\begin{align*}
\varphi(x)\coloneqq\frac{d\mu}{d{\rm Leb}}(x)
=\begin{cases}
2^{n}\varphi_-(2^{n-1}x) & x\in I_{n}^-\ (n\in\mathbb{N}),\\
\varphi_0(x) & x\in I_0,\\
2^{n}\varphi_+(2^{n-1}x-2^{n-1}+1) & x\in I_n^+\ (n\in\mathbb N).
\end{cases}
\end{align*}
\end{remark}

\section{Proof of Theorem \ref{thm:main}}
\subsection{Skew-product representation}
This subsection is devoted to prepare the skew-product transformation corresponding to our random dynamical system.
The proof of our main result is based on the application of the Thaler--Zweim\"{u}ler theorem to the  skew-product transformation. 

We let
$(\Omega,\mathcal B (\Omega),\mathbb P)$ be a probability space with $\Omega=\{0,1\}^{\mathbb N}$, the Borel field of $\Omega$, and
\begin{align*}
\mathbb P \left(\omega=\{\omega_1, \omega_2, \cdots\}\in\Omega : \omega_1=i_1,\cdots,\omega_n=i_n\right) = \frac{1}{2^n}
\end{align*}
for any $i_1,\cdots,i_n\in\{0,1\}$ and $n\in\mathbb N$.
For notational simplicity, we denote 
$
B^\sim=\Omega\times B
$
 for each $B\in\mathcal B([0,1])$, and equip $[0,1]^\sim = \Omega\times[0,1]$ with the product topology of $\Omega$ and $[0,1]$. We write $\mathcal B([0,1]^\sim)$ for the Borel field of $[0,1]^\sim$. 
Then we define the skew-product transformation on $[0,1]^{\sim}$, $\widetilde{T}:[0,1]^\sim\to[0,1]^\sim$ by
\begin{align*}
\widetilde{T}(\omega,x) = (\theta\omega, f_{\omega_1}(x))
\end{align*}
for
$\omega = (\omega_1, \omega_2, \cdots)\in\Omega$ and $x\in[0,1],$
where $\theta:\Omega\to\Omega$ is the left shift map $\theta(\omega_1, \omega_2, \cdots)=(\omega_2, \omega_3, \cdots)$.
Note that from the definition of $\widetilde{T}$ we have
\begin{align}\label{eq:0401eb}
\widetilde{T}^n(\omega,x)=(\theta^n\omega,  f_\omega^{(n)}(x)), \quad f_\omega^{(n)} =f_{\omega_n}\circ\cdots\circ f_{\omega_1}
\end{align}
for each $n\in\mathbb{N}$.
Given $n\leq m$ and $\omega =(\omega _1, \ldots ,\omega _m)\in \{0,1\} ^m$, we still use the notation $f^{(n)}_\omega$ to denote $f_{\omega_n}\circ\cdots\circ f_{\omega_1}$, so that the invariance of $\mu$ for $T$ implies 
\begin{equation}\label{eq:0218}
\mathbb{E}\left[\mu(T^{(-n)}A)\right] = \frac{1}{2^n}\sum _{\omega\in\{0,1\}^n} \mu \left(f_\omega ^{(-n)}A\right) = \mu (A)
\end{equation} 
for every Borel set $A$ and $n\in \mathbb N$, where $f^{(-n)}_\omega(A)$ is the inverse image of $A$ by $f^{(n)}_\omega$. 
For simplicity, we write $f_\omega^{(0)}$ for the identity map for each $\omega$. Similarly, we define $h_\omega ^{(n)}$ for  $n\in\mathbb N$ and $\omega\in\{0,1\}^n$,  from the measurable maps $h_0$, $h_1$ on $Y$. 

Let $\widetilde{\mu}\coloneqq\mathbb P\otimes\mu$ where $\mu$ is a $\sigma$-finite and infinite $T$-invariant measure given in the equation (\ref{inv2-2}) under the assumptions of Lemma \ref{lem2-3}, that is, $\nu$ is a $T_{\mathrm{core}}$-invariant probability measure and $\nu(I_1^-)\nu(I_1^+)>0$.
Then it is well known (see \cite{Kifer86} for example) that $\widetilde{\mu}$ is also an invariant measure for the skew-product transformation $\widetilde{T}$ which is  $\sigma$-finite and infinite.
From Lemma \ref{lem2-3}, we also have for each $\epsilon\in(0,\frac{1}{2})$,
\begin{align*}
\widetilde{\mu}\left( (\epsilon,1-\epsilon)^{\sim} \right)<\infty\text{ and }\
\widetilde{\mu}\left([0,\epsilon)^{\sim}\right)=
\widetilde{\mu}\left((1-\epsilon,1]^{\sim}\right)=\infty.
\end{align*}

Furthermore, we have the following   lemmas about the annealed metrical transitivity of $(T, \mu)$ and the ergodicity of $(\widetilde T, \widetilde{ \mu})$.
Recall that a random dynamical system $(T,\mu)$ is called \emph{ergodic} if   any $T$-invariant set $A\in\mathcal{B}([0,1])$ in the sense that $\mathbb{E}[1_A\circ T]=1_A$ is either $A=\emptyset$ or $[0,1]\setminus A=\emptyset\pmod{\mu}$.

\begin{lemma}\label{lem-2}
If $(T,\mu)$ is annealed metrically transitive, then $(T,\mu)$ is ergodic.
\end{lemma}

\begin{proof}
Let $A\in\mathcal{B}$ be a $T$-invariant set.
That is, $A$ satisfies $1_A=\mathbb{E}[1_{T^{-1}A}] =\frac{1}{2}(1_{f_0^{-1}A}+1_{f_1^{-1}A})$.
Then $1_A=1_{f_0^{-1}A\cap f_1^{-1}A}+\frac{1}{2}1_{f_0^{-1}A\triangle f_1^{-1}A}$ and this implies $f_0^{-1}A=f_1^{-1}A=A\pmod{\mu}$.
From the annealed metrical transitivity of $(T,\mu)$, for any sets of $\mu$-positive measure, say $B,C$, we can find $n\ge1$ such that $\mathbb{E}[\mu(T^{(-n)}B\cap C)]>0$.
But if we take $B=A$ and $C=A^c$ above where $A$ is a $T$-invariant set, then we have $\mathbb{E}[\mu(T^{(-n)}B\cap C)]=\mu(A\cap A^c)=0$ and hence either $A=\emptyset$ or $[0,1]\pmod{\mu}$.
\end{proof}

\begin{lemma}\label{lem-3}
$(T,\mu)$ is ergodic if and only if $(\widetilde{T},\widetilde{\mu})$ is ergodic.
\end{lemma}

\begin{proof}
See Theorem 2.1 in \cite[Chapter 1]{Kifer86}.
\end{proof}

Under the help of Lemmas \ref{lem-2} and \ref{lem-3}, we get the following assertion.

\begin{theorem}\label{thm: ce}
If $(T_{\mathrm{core}},\nu)$ is annealed metrically transitive, then $(\widetilde T,\widetilde{\mu})$ is conservative and ergodic.
\end{theorem}

We prepare a   lemma before proving Theorem \ref{thm: ce}.

\begin{lemma}\label{lem-1}
Let $n\ge 2$.
Then for $\mathbb{P}$-almost all $\omega=(\omega_1,\omega_2,\dots)\in\Omega$, there exists $N_-\in\mathbb{N}$ (resp. $N_+\in\mathbb{N}$) such that for any $x\in I_n^{-}$ (resp.\ $I_n^+$), we have
\begin{align*}
f_{\omega}^{(N_-)}(x)
\in I_1^{-}\text{ (resp.\ $f_{\omega}^{(N_+)}(x) \in I_1^+$)}.
\end{align*}
\end{lemma}

\begin{proof}
We fix $n\ge2$.
Note that $\mathbb{P}(T=f_0)=\mathbb{P}(T=f_1)=\frac{1}{2}$ and $f_0\circ f_1(x)=f_1\circ f_0(x)=x$ for any $x\in I_n^{\pm}$ where $n\ge2$.
Then setting for $\omega=(\omega_1,\omega_2,\dots)\in\Omega$ and $k\ge1$
\begin{align*}
\eta_k(\omega)=
\begin{cases}
-1 & (\omega_k=0)\\
1 & (\omega_k=1)
\end{cases}
\end{align*}
and
\begin{align}\label{eq:rw}
W_k^{-n}(\omega)=-n+\sum_{i=1}^{k}\eta_i(\omega),
\end{align}
we identify our random dynamics $(T,\mu)$ with a simple random walk on $\mathbb{Z}\setminus\{-1,0,1\}$ outside $Y=I_1^-\cup I_0\cup I_1^+$.
Let $\Omega_{-n}=\{\omega\in\Omega:\varphi_{-1}(\omega)<\infty\}$ where $\varphi_{-1}(\omega)=\inf\{j\ge1:W_j^{-n}(\omega)=-1\}$ stands for the first hitting time of $-1$ for $(W_k^{-n})_{k}$.
Then from a standard argument for one-dimensional random walk which is recurrent and irreducible, we have $\mathbb{P}(\Omega_{-n})=1$, i.e., for $\mathbb{P}$-almost every $\omega\in\Omega$, there is $N=N(\omega)\in\mathbb{N}$ such that $f_{\omega}^{(N)}(I_n^-)=I_1^-$.
By the symmetry, we conclude the result is true for the positive side $\cup_{n\ge2}I_n^+$ and complete the proof.
\end{proof}

\begin{proof}[Proof of Theorem \ref{thm: ce}]
We first show that $(\widetilde{T},\widetilde{\mu})$ is conservative.
To show this, by Maharam's recurrence theorem (Theorem 1.1.7 in \cite{Aaronson1997}), it is enough to see the existence of a set $M^{\sim}\in\mathcal{B}([0,1]^{\sim})$ such that $\widetilde{\mu}(M^{\sim})<\infty$ and $\bigcup_{n=0}^{\infty}\widetilde{T}^{-n}M^{\sim}=[0,1]^{\sim}\pmod{\widetilde{\mu}}$.
For this, we set $M^{\sim}\coloneqq\Omega\times Y$ where $Y=I_1^{-1}\cup I_0\cup I_1^+$.
Then we have $\widetilde{\mu}(M^{\sim})<\infty$ by construction of $\mu$ in (\ref{inv2-2}).
It also holds from Lemma \ref{lem-1} that $\bigcup_{n=0}^{\infty}\widetilde{T}^{-n}M^{\sim}=[0,1]^{\sim}\pmod{\widetilde{\mu}}$.

For showing the ergodicity of $(\widetilde{T},\widetilde{\mu})$, we need to prove the annealed metrical transitivity of $(T,\mu)$ from Lemma \ref{lem-2} and Lemma \ref{lem-3}.
We take arbitrary Borel sets $A$ and $B$ with $\mu(A)\mu(B)>0$.
We first consider the case when $\mu(B\cap I_0) <\mu(B)$, that is, $\mu(B\cap (\bigcup _{n\geq 1} I_n^-\cup I_n^+)) >0$. 
If $\mu(B\cap I_{n_2}^-) >0$ (resp., $\mu(B\cap I_{n_2}^+) >0$) for some $n_2\geq 1$, then
we let  $B'=B\cap I_{n_2}^-$ and $\omega '=(0\cdots 0) $  (resp., $B'=B\cap I_{n_2}^+$ and $\omega '=(1\cdots 1) $). 
Define $A' \subset A$, an integer $n_1\ge 1$ and $\omega \in \{0,1\}^{n_1}$ in a similar manner, except the case $\mu (A\cap I_0) =\mu (A)$ in which we set $A'=A$ and $n_1=0$.
Then, observing 
that $f_{(0\cdots 0)}^{(-j)} (I_{n}^-) = f_1^{j}(I_{n}^-) = I_{n-j}^-$   and $f_{(1\cdots 1)}^{(-j)} (I_{n}^+) =  f_0^{j}(I_{n}^+)=I_{n-j}^+$ for each $j=0,\ldots ,n-1$ and $n\geq 1$ (recall that $f_0 \circ f_1(x) =f_1\circ f_0(x) =x$ on $Y^c$) and that $f_0(I_1^+),f_1(I_1^-)\subset Y$, we get 
\begin{equation}\label{eq:0218b}
\mu\left(f^{(-n_1)}_{\omega}A' \cap Y\right) >0, \quad  f^{(-(n_2-1))}_{\omega'} B' \subset I_1^- \cup I_1^+  
\end{equation}
and
\begin{equation}\label{eq:0218c}
 f^{(-j)}_{\omega'} B' \cap Y=\emptyset \quad \text{for each $j=0,\ldots ,n_2-2$}.
\end{equation}
Set $A''\coloneqq f_{\omega}^{(-n_1)}A'$ and $B''\coloneqq f_{\omega'}^{(-(n_2-1))}B'$.  
Then, since $\nu (E) \geq \frac{\mu (E)}{2}$ for each $E\in \mathcal B(Y)$ and \eqref{eq:0218b}, $\nu(A'')\nu(B'')>0$.
Therefore, by the assumption that $(T_{\mathrm{core}},\nu)$ is annealed metrically transitive,  
we can find $n_3\ge 0$ such that
\begin{equation}\label{eq:0218e}
\mathbb{E}\left[\nu\left(T_{\mathrm{core}}^{(-n_3)}A''\cap B''\right)\right]
= \frac{1}{2^{n_3}} \sum _{\eta \in\{0,1\}^{n_3}}\nu\left(h_{\eta }^{(-n_3)}A''\cap B''\right) >0.
\end{equation}
In particular, there exists $\omega ''\in\{0,1\}^{n_3}$ satisfying
$\nu( 
h_{\omega ''}^{(-n_3)}A''\cap B'')>0$.
It follows from $\nu(E)\le\mu(E)$ for each $E\in\mathcal{B}(Y)$ that $\mu(h_{\omega ''}^{(-n_3)}A''\cap B'')>0$.
Recall that $h_0=h_1=f_1$ on $I_1^-$, $h_0=f_0$, $h_1=f_1$ on $I_0$, $h_0=h_1=f_0$ on $I_1^+$.
Hence, we have
$\mu(f_{\omega'''}^{(-n_3)}A''\cap B'')>0$ for some $\omega'''\in\{0,1\}^{n_3}$. 
Denote $(\omega _{n_2-1}'+1 \pmod 1, \ldots, \omega_1'+1 \pmod 1)$ 
by $\widehat\omega'$, then we have $ f_{\widehat \omega'}^{(-(n_2-1))} (C) = f_{\omega'}^{(n_2-1)} (C)$ for any $C\subset I_1^- \cup I_1^+$  
 due to \eqref{eq:0218b} and \eqref{eq:0218c}.
In particular, $ f_{\widehat \omega'}^{(-(n_2-1))} \circ f^{(-(n_2-1))}_{\omega'}(B')= B'$, and thus 
we get 
\begin{align*}
\mu\left( f_{\widehat \omega ' \omega ''' \omega}^{(-(n_2-1+n_3+n_1))} A\cap B \right)
\ge&\mu\left( f_{\widehat \omega '}^{(-(n_2-1))}  \left( f_{\omega ''' \omega}^{(-(n_3+n_1))} A'\cap  f_{\omega '}^{(-(n_2-1))} B'\right) \right)\\ 
=&\mu\left( f_{\omega '}^{(n_2-1)}  \left( f_{\omega ''' }^{(- n_3)} A''\cap  B''\right) \right).
\end{align*}
On the other hand, for $D\coloneqq f_{\omega ''' }^{(- n_3)} A''\cap  B''$,  it follows from \eqref{eq:0218} that
\begin{align*}
\mu\left( f_{\omega '}^{(n_2-1)}  D \right) &=
\frac{1}{2^{n_2-1}}\sum _{\eta\in\{0,1\}^{n_2-1}} \mu\left( f_{\eta}^{(-(n_2-1))} \circ f_{\omega '}^{(n_2-1)}   D \right)\\
&\geq 
\frac{1}{2^{n_2-1}} \mu\left( f_{\omega '}^{(-(n_2-1))} \circ f_{\omega '}^{(n_2-1)} D \right)\\ 
&=\frac{1}{2^{n_2-1}}\mu\left(  f_{\omega ''' }^{(- n_3)} A''\cap  B'' \right)
>0. \color{black}
\end{align*}
Therefore, with $N\coloneqq n_2-1+n_3+n_1$,  we have
\begin{align}\label{eq:0218f}
\mathbb{E}\left[\mu\left(T^{(-N)} A \cap B\right)\right] &= \frac{1}{2^{N}} \sum _{\eta \in \{0,1\} ^{N}} \mu\left( f_{\eta}^{(-N)} A\cap B \right)\\\notag
&\ge \mu\left( f_{\widehat \omega ' \omega ''' \omega}^{(- N)} A\cap B \right)>0.
\end{align}
This shows the desired result for the first case.
In particular, if we take $B=I_1^-$ above, then we have for any $A\in\mathcal{B}([0,1])$ with $\mu(A)>0$,
\begin{align}\label{eq:219}
\mu\left( f_{\omega}^{(-N)} A\cap I_1^- \right)
>0
\end{align}
for some $\omega\in\Omega$ and $N\in\mathbb{N}$.

We second consider the case when $\mu (B\cap I_0)=\mu(I_0)$, implying $\nu(B) \geq \nu (B\cap I_0)=\mu(I_0)>0$.
From \eqref{eq:219}, we get
\[
\nu\left( f_{\omega}^{(- N)} A \right) \geq \frac{1}{2}\mu\left( f_{\omega}^{(- N)} A\cap I_0^-\right)>0
\]
for some $\omega\in\Omega$ and $N\in\mathbb{N}$.
Hence, 
 it follows from the annealed metrical transitivity of $(T_{\mathrm{core}},\nu)$ that 
  there exists an integer $m\geq 0$ and $\widetilde\omega $ such that $\nu (f^{(-m)}_{\widetilde\omega } \circ f_{\omega}^{(- N)} A \cap B)>0$.
Since $\nu \leq \mu$ on $Y$, this implies  
 $\mu (f_{\widetilde\omega \omega}^{(- (m+N))} A \cap B)>0$,
which concludes  the annealed metrical transitivity of $(T,\mu)$ by repeating the argument in \eqref{eq:0218f}.
\end{proof}

\subsection{The end of the proof of Theorem \ref{thm:main}}
Now we   complete the proof of Theorem \ref{thm:main} by checking the assumptions (C1)--(C4) in Theorem \ref{thm:t} with respect to 
\begin{equation}\label{eq:0401ea}
\tau = \widetilde T, \; X=[0,1]^{\sim}, \; m=\widetilde{\mu}, \; A=Y^\sim, \; A^{\pm}=\Big(\bigcup_{k=2}^{\infty}I_k^{\pm}\Big)^\sim
\end{equation}
(recall that $B^\sim=\Omega\times B$ for each $B\in \mathcal B([0,1])$). 
We set 
$(\widetilde{Y})_n=(Y^\sim)^c\cap\{\varphi_{Y^\sim}=n\}$, 
where $\varphi_{Y^\sim}(\omega,x)=\inf\{n\ge1 : \widetilde{T}^n(\omega,x)\in Y^\sim\}$, so that
\begin{align*}
(\widetilde{Y})_n&=\bigcup_{k=1}^{\infty}\left(\Omega_{n,k}^-\times I_{k+1}^-\right)
\cup\bigcup_{k=1}^{\infty}\left(\Omega_{n,k}^+\times I_{k+1}^+\right)
\end{align*}
where
\begin{align*}
\Omega_{n,k}^-&\coloneqq\bigg\{\{\omega_1,\omega_2,\dots\}\in\Omega : \\
&\quad\qquad\#\{1\le i<l : \omega_i=1\}-\#\{1\le i<l : \omega_i=0\}<k\text{ for }l<n,\\
&\quad\qquad\qquad\qquad\#\{1\le i\le n : \omega_i=1\}-\#\{1\le i\le n : \omega_i=0\}=k
\bigg\}
\end{align*}
and
\begin{align*}
\Omega_{n,k}^+&\coloneqq\bigg\{\{\omega_1,\omega_2,\dots\}\in\Omega : \\
&\quad\qquad\#\{1\le i<l : \omega_i=0\}-\#\{1\le i<l : \omega_i=1\}<k\text{ for }l<n,\\
&\quad\qquad\qquad\qquad\#\{1\le i\le n : \omega_i=0\}-\#\{1\le i\le n : \omega_i=1\}=k
\bigg\}.
\end{align*}
Each $\omega\in\Omega_{n,k}^{\pm}$ depends only on the first $n$-coordinates so that we can write $\Omega_{n,k}^{\pm}=\Omega_{n,k}^{\pm}\mid_n\times\Omega$ for some $\Omega_{n,k}^{\pm}\mid_n\subset\{0,1\}^n$.

In what follows, $(W_j^{k})_j=(W_j^{k}(\omega))_j$ denotes the random walk which starts from $k\in\mathbb{Z}\setminus\{-1,0,1\}$ given in \eqref{eq:rw} and for $l\in\mathbb{Z}$ with $lk>0$, $\varphi_{l}^k=\varphi_{l}^k(\omega)=\inf\{j\ge1:W_j^k=l\}$ denotes the first hitting time of $l$ for the random walk $(W_j^k)_j$.
Then we have for $k\ge1$
\begin{align}\label{eq2-4}
\mathbb{P}\left(\varphi_{1}^{k+1}=n\right)=\mathbb{P}\left(\varphi_{-1}^{-(k+1)}=n\right)
=\mathbb{P}\left(\Omega_{n,k}^{\pm}\right)=\frac{\#\left(\Omega_{n,k}^{\pm}\mid_n\right)}{2^n}.
\end{align}
For these notation, we have the following lemma.

\begin{lemma}\label{lem2-6}
For each $n\ge1$,
 it holds that
\begin{align*}
\mathcal{L}_{\widetilde{T}}^{n} 1_{(\widetilde{Y})_{n} \cap A^-}(\omega, x) &=c_{n} 1_{I_1^-}(x) \\
\mathcal{L}_{\widetilde{T}}^{n} 1_{(\widetilde{Y})_{n}}(\omega, x) &=c_{n} 1_{I_1^-\cup I_1^+}(x)
\end{align*}
where
\begin{align*}
c_{n}=\sum_{k = 1}^{\infty} \mathbb{P}\left(\varphi_{1}^{k+1}=n\right).
\end{align*}
\end{lemma}

\begin{proof}
We prove this lemma in showing
\begin{align*}
\int_{\widetilde{I}} \psi(\omega, x) \mathcal{L}_{\widetilde{T}}^{n} 1_{(\widetilde{Y})_{n} \cap A^\pm}(\omega, x) d \widetilde{\mu}=\int_{\widetilde{I}} \psi(\omega, x) c_{n} 1_{I_1^{\pm}}(x) d \widetilde{\mu}
\end{align*}
for $\psi\in L^\infty(\widetilde{I},\widetilde{\mu})$ of the form $\psi(\omega,x)=\psi_1(\omega)\psi_2(x)$.
Write
\begin{align*}
&\int_{\widetilde{I}} \psi(\omega, x) \mathcal{L}_{\widetilde{T}}^{n} 1_{(\widetilde{Y})_{n} \cap A^-}(\omega, x) d \widetilde{\mu} \\ 
=& \int_{\widetilde{I}} \psi_{1}\left(\theta^{n} \omega\right) \psi_{2}\left(f_{\omega_{n}} \circ f_{\omega_{n-1}} \circ \cdots \circ f_{\omega_{1}} x\right) 1_{(\widetilde{Y})_{n}}(\omega, x) \sum_{k = 1}^{\infty} 1_{I_{k+1}^-}(x) d \widetilde{\mu} \\
=&\sum_{k=1}^{\infty} \int_{\Omega}\psi_1\left(\theta^n\omega\right){\sum_{(\omega_1,\dots,\omega_n)\in\Omega_{k,n}^-\mid_n}}1_{[\omega_1\dots\omega_n]}(\omega)d\mathbb{P}(\omega) \int_I\psi_2\left(f_1^kx\right) 1_{I_{k+1}^-}(x)d\mu(x)
\end{align*}
since $f_{\omega}^{(n)}=f_1^k$ for $(\omega_1,\dots,\omega_n)\in\Omega_{k,n}^-\mid_n$.
For $\mathcal{L}_{\theta}$ the Perron--Frobenius operator of $\theta$ with respect to $\mathbb{P}$, we have $\mathcal{L}_{\theta}^n 1_{[\omega_1,\dots,\omega_n]}=2^{-n}\cdot1_{\Omega}$ for each $(\omega_1,\dots,\omega_n)\in\{0,1\}^n$.
We also have $I_{k+1}^-=f_1^{-k}I_1^-$ and
\begin{align*}
\frac{d\left((f_1^k)_*\mu\right)}{d\mu}\bigg\vert_{I_1^-}=
\frac{d\left(2\nu_-\circ f_0^{-k}\circ f_1^{-k}\right)}{d\mu}\bigg\vert_{I_1^-}=
\frac{d(2\nu_-)}{d\mu}\bigg\vert_{I_1^-}=
\frac{d\mu}{d\mu}\bigg\vert_{I_1^-}
=1
\end{align*}
by the construction of $\mu$ and the fact that $f_0^{-k}\circ f_1^{-k}=\mathrm{id}$ on $I_1^-$.
Hence, by the equation \eqref{eq2-4}, it holds that
\begin{align*}
\int_{\widetilde{I}} \psi(\omega, x) \mathcal{L}_{\widetilde{T}}^{n} 1_{(\widetilde{Y})_{n} \cap A^-}(\omega, x) d \widetilde{\mu}
=&\sum_{k=1}^{\infty} \mathbb{P}\left(\varphi_{-1}^{-(k+1)}=n\right) \int_{\Omega}\psi_1d\mathbb{P} \int_I\psi_21_{I_1^-}d\mu \\
=& \int_{\widetilde{I}} \psi(\omega,x) \sum_{k=1}^{\infty} 1_{I_1^-}(x) \mathbb{P}\left(\varphi_{-1}^{-(k+1)}=n\right) d \widetilde{\mu}.
\end{align*}
For the right half part, the calculation of $\mathcal{L}_{\widetilde{T}}^{n} 1_{(\widetilde{Y})_{n} \cap A^+}$ is similar and is omitted.
\end{proof}

From the standard argument of simple symmetric random walks on $\mathbb{Z}$ (cf.~\cite{HY2021}), for $0<s<1$ we have
\begin{align*}
\sum_{n=1}^{\infty}s^n \mathbb{P}\left(\varphi_{-1}^{-(k+1)}=n\right) =\left(\frac{1-\sqrt{1-s^{2}}}{s}\right)^{k}
\end{align*}
for each $k\in\mathbb{N}$.
Thus, we have
\begin{align*}
\sum_{n=1}^{\infty} c_{n} s^{n} &=\sum_{n=1}^{\infty} \sum_{k=1}^{\infty} s^{n}\mathbb{P}\left(\varphi_{-1}^{-(k+1)}=n\right) \\
&=\sum_{k=1}^{\infty}\left(\frac{1-\sqrt{1-s^{2}}}{s}\right)^{k} \\
&= \frac{\left(1-\sqrt{1-s^{2}}\right)}{\sqrt{1-s^{2}}-1+s} \\
&\sim \frac{1}{\sqrt{2(1-s)}} \text { as } s \uparrow 1.
\end{align*}
Here, $a_N\sim b_N$ denotes $\frac{a_N}{b_N}\to 1$ as $N\to\infty$. 
Then by Karamata's Tauberian Theorem (see Proposition 4.2 in \cite{TZ2006}) we conclude
\begin{align*}
\sum_{n=1}^{N} c_{n} \sim \frac{1}{\sqrt{2}\Gamma\left(\frac{3}{2}\right)}  N^{\frac{1}{2}}= \sqrt{\frac{2}{\pi}} N^{\frac{1}{2}}.
\end{align*}
From the above computation together with Lemma \ref{lem2-6}, we can calculate asymptotics of the wandering rate:
\begin{align*}
w_{N}(Y^\sim)
&=\int_{Y^\sim} \sum_{n=0}^{N-1} \mathcal{L}_{\widetilde{T}}^{n} 1_{(\widetilde{Y})_{n}} d \widetilde{\mu}\\ 
&=\mu(Y) + \left(\mu\left(I_1^-\right)+\mu\left(I_1^+\right)\right)\sum_{n=1}^{N-1} c_{n} \\
&\sim \left(\mu\left(I_1^-\right)+\mu\left(I_1^+\right)\right) \sqrt{\frac{2}{\pi}} N^{\frac{1}{2}}
\end{align*}
and hence the condition (C2) is valid with the index $\frac{1}{2}$.
We also have
\begin{align*}
w_{N}\left(Y^\sim, A^-\right)
&=\widetilde{\mu}\left(Y^\sim \cap \widetilde{T}^{-1} A^-\right)+\sum_{n=1}^{N-1} \int_{\widetilde{I}}\mathcal{L}_{\widetilde{T}}^{n} 1_{\tilde{Y}_{n} \cap A^-}d \widetilde{\mu} \\
&=\frac{1}{2}\mu\left(I_1^-\right)+\mu\left(I_1^-\right) \sum_{n=0}^{N-1} c_{n}.
\end{align*}
Therefore, it also follow from Lemma \ref{lem2-6} that
\begin{align*}
\lim_{N\to\infty}\frac{1}{w_{N}(Y^\sim)} \sum_{n=0}^{N-1} \mathcal{L}_{\widetilde{T}}^{n} 1_{(\widetilde{Y})_{n}}
&= \lim_{N\to\infty}\frac{\sum_{n=0}^{N-1}c_n1_{I_1^-\cup I_1^+}}{w_N(Y^\sim)}\\
&= \frac{1}{\mu\left(I_1^-\right)+\mu\left(I_1^+\right)}\cdot 1_{I_1^-\cup I_1^+}
\end{align*}
uniformly on $Y^\sim$
and
\begin{align*}
\lim_{N\to\infty}\frac{1}{w_{N}\left(Y^\sim, A^-\right)} \sum_{n=0}^{N-1} \mathcal{L}_{\widetilde{T}}^{n} 1_{(\widetilde{Y})_{n} \cap A^-}
&= \frac{1}{\mu(I_1^-)}\cdot 1_{I_1^-}
\end{align*}
uniformly on $Y^\sim$
and we have checked the conditions (C1) and (C3).

Finally, for the condition (C4), we have
\begin{align*}
\frac{w_{N}\left(Y^\sim, A^-\right)}{w_{N}(Y^\sim)}
&=\frac{ \mu\left(I_1^-\right)\sum_{n=0}^{N-1} c_{n}+\frac{1}{2}\mu\left(I_1^-\right)}{\left(\mu\left(I_1^-\right)+\mu\left(I_1^+\right)\right) \sum_{n=0}^{N-1} c_{n}}\\
&\to\frac{\mu\left(I_1^-\right)}{\mu\left(I_1^-\right)+\mu\left(I_1^+\right)}=\frac{\nu\left(I_1^-\right)}{\nu\left(I_1^-\right)+\nu\left(I_1^+\right)}
\end{align*}
as $N$ tends to infinity.

Now we can apply Theorem \ref{thm:t} to the specific 
%\marginpar{\tiny YN: なぜここから結論がimmediateか一言書く}
%The conclusion of Theorem \ref{thm:main} immediately follows from the conclusion of  Theorem \ref{thm:t} for such
 $\tau, X, m, A, A^\pm$  given in \eqref{eq:0401ea}. 
For any random variable $\Theta$ with values in $Y$ which is independent of $\{ T_n\} _{n=1}^\infty$ and whose  distribution  $m_\Theta$ is absolutely continuous with respect to $\nu$ (so that $m_\Theta$ is absolutely continuous with respect to $\mu$ by construction of $\mu$), the probability measure $\mathbb P\times m_\Theta$ is absolutely continuous with respect to $\widetilde{\mu}$.
Hence, the equation \eqref{eq:0401eb} and the strong distributional convergence  in Theorem \ref{thm:t} with $\alpha = \frac{1}{2}$ and $B=[\frac{1}{2}, 1]^\sim$ complete the proof of Theorem \ref{thm:main}.

\begin{remark}\label{rem:211}
In order to apply Thaler--Zweim{\"u}ller's Darling--Kac law (\cite[Theorem 3.1]{TZ2006}), one needs to compute
\[
\widehat a_N:= \frac{4}{\pi}\cdot\frac{N}{w_N(Y^\sim)}
\]
for which
\[
\frac{1}{\widehat a_N} \sum_{n=0}^{N-1}1_{\{T^{(n)}(\Theta)\in E\}}
\xrightarrow[N\to\infty]{\mathrm{d}}
 \sqrt{\frac{\pi}{2}}  \mu (E)  \vert \mathcal N\vert
\]
holds for each Borel set $E$.
On the other hand, it holds that
\begin{align*}
\widehat a_N&\sim
\frac{4}{\pi}\sqrt{\frac{\pi}{2}}\cdot\frac{1}{\sqrt{N}}\cdot\frac{N}{\mu(I_1^-)+\mu(I_1^+)}
=2\sqrt{\frac{2}{\pi}}\frac{\sqrt{N}}{\mu(I_1^-)+\mu(I_1^+)}
\quad\text{as }N\to\infty .
\end{align*}
Hence, we immediately get the claim in Remark \ref{rem:DK}. 
\end{remark}

\section*{Acknowledgments}
We are sincerely grateful to Toru Sera for many fruitful discussions and comments, which substantially helped improving the earlier versions of the present paper. 
Yushi Nakano was  supported by JSPS KAKENHI Grant Numbers 	18K03376, 19K14575, 19K21834, 21K03332. 
Fumihiko Nakamura was  supported by JSPS KAKENHI Grant Number 	  19K21834.  
Hisayoshi Toyokawa was  supported by JSPS KAKENHI Grant Numbers 19K21834, 21K20330. 
Kouji Yano was  supported by JSPS KAKENHI Grant Numbers  JP19H01791, 19K21834 and    by JSPS Open Partnership Joint Research Projects grant Grant Number JPJSBP120209921.
This research was supported by RIMS.

  %%%%% References
%\begin{thebibliography}{99}

%  \bibitem{}
  
%\end{thebibliography}

\end{document}